\documentclass[twoside, a4paper, 10pt]{amsart}
\title[ ]{Arithmetic subtrees in large subsets of products of trees}
\usepackage{amsfonts}
\usepackage{amsthm}
\usepackage{verbatim}
\usepackage{amsmath, amssymb}
\usepackage{tikz}
\usetikzlibrary{matrix, arrows}
\usepackage{listings}
\usepackage{color}
\usepackage{listings}
\usepackage[all]{xy}
\usepackage[pdftex,colorlinks,linkcolor=blue,citecolor=blue]{hyperref}
\usepackage{graphicx}
\usepackage{float}
\usepackage[margin=3cm]{geometry}
\usepackage{bigints}
\usepackage{dsfont}
\author{Kamil Bulinski}
\address{School of Mathematics and Statistics, University of Sydney, Australia}
\email{kamil.bulinski@sydney.edu.au}

\author{Alexander Fish}
\address{School of Mathematics and Statistics, University of Sydney, Australia}
\email{alexander.fish@sydney.edu.au}

\begin{document}
\maketitle
\raggedbottom

%% Mathcal large
\newcommand{\cA}{\mathcal{A}}
\newcommand{\cB}{\mathcal{B}}
\newcommand{\cC}{\mathcal{C}}
\newcommand{\cD}{\mathcal{D}}
\newcommand{\cE}{\mathcal{E}}
\newcommand{\cF}{\mathcal{F}}
\newcommand{\cG}{\mathcal{G}}
\newcommand{\cH}{\mathcal{H}}
\newcommand{\cI}{\mathcal{I}}
\newcommand{\cJ}{\mathcal{J}}
\newcommand{\cK}{\mathcal{K}}
\newcommand{\cL}{\mathcal{L}}
\newcommand{\cM}{\mathcal{M}}
\newcommand{\cN}{\mathcal{N}}
\newcommand{\cO}{\mathcal{O}}
\newcommand{\cP}{\mathcal{P}}
\newcommand{\cQ}{\mathcal{Q}}
\newcommand{\cR}{\mathcal{R}}
\newcommand{\cS}{\mathcal{S}}
\newcommand{\cT}{\mathcal{T}}
\newcommand{\cU}{\mathcal{U}}
\newcommand{\cV}{\mathcal{V}}
\newcommand{\cW}{\mathcal{W}}
\newcommand{\cX}{\mathcal{X}}
\newcommand{\cY}{\mathcal{Y}}
\newcommand{\cZ}{\mathcal{Z}}
%% Mathbb large
\newcommand{\bA}{\mathbb{A}}
\newcommand{\bB}{\mathbb{B}}
\newcommand{\bC}{\mathbb{C}}
\newcommand{\bD}{\mathbb{D}}
\newcommand{\bE}{\mathbb{E}}
\newcommand{\bF}{\mathbb{F}}
\newcommand{\bG}{\mathbb{G}}
\newcommand{\bH}{\mathbb{H}}
\newcommand{\bI}{\mathbb{I}}
\newcommand{\bJ}{\mathbb{J}}
\newcommand{\bK}{\mathbb{K}}
\newcommand{\bL}{\mathbb{L}}
\newcommand{\bM}{\mathbb{M}}
\newcommand{\bN}{\mathbb{N}}
\newcommand{\bO}{\mathbb{O}}
\newcommand{\bP}{\mathbb{P}}
\newcommand{\bQ}{\mathbb{Q}}
\newcommand{\bR}{\mathbb{R}}
\newcommand{\bS}{\mathbb{S}}
\newcommand{\bT}{\mathbb{T}}
\newcommand{\bU}{\mathbb{U}}
\newcommand{\bV}{\mathbb{V}}
\newcommand{\bW}{\mathbb{W}}
\newcommand{\bX}{\mathbb{X}}
\newcommand{\bY}{\mathbb{Y}}
\newcommand{\bZ}{\mathbb{Z}}

\newcounter{dummy} \numberwithin{dummy}{section}

\theoremstyle{definition}
\newtheorem{mydef}[dummy]{Definition}
\newtheorem{prop}[dummy]{Proposition}
\newtheorem{corol}[dummy]{Corollary}
\newtheorem{thm}[dummy]{Theorem}
\newtheorem{lemma}[dummy]{Lemma}
\newtheorem{eg}[dummy]{Example}
\newtheorem{notation}[dummy]{Notation}
\newtheorem{remark}[dummy]{Remark}
\newtheorem{claim}[dummy]{Claim}
\newtheorem{Exercise}[dummy]{Exercise}
\newtheorem{question}[dummy]{Question}
\newtheorem{conjecture}[dummy]{Conjecture}

\begin{abstract}
Furstenberg-Weiss have extended Szemer\'edi's theorem on arithmetic progressions to trees by showing that a large subset of the tree contains arbitrarily long arithmetic subtrees. We study higher dimensional versions that analogously extend the multidimensional Szemer\'edi theorem by demonstrating the existence of certain arithmetic structures in large subsets of a cartesian product of trees.
\end{abstract}

\section{Introduction}

\subsection{Arithmetic subtrees} In \cite{FWTrees}, Furstenberg and Weiss have generalized Szemer\'edi's Theorem on arithmetic progressions to trees. We begin by reformulating this theorem in terms of semigroups as follows. We let $\Lambda$ be a finite set and let $\Lambda^* = \bigcup_{n=0}^{\infty} \Lambda^n$ denote the set of all finite words in the alphabet $\Lambda$. We will write an element of $\Lambda^*$ as a word $\lambda_1 \cdots \lambda_n$ rather than as a tuple. Note that $\Lambda^*$ is the free semigroup generated by $\Lambda$, with the empty word $\emptyset$ being the identity. We let $|w| \in \bZ_{\geq 0}$ denote the length of $w \in \Lambda^*$ and we let $B_r(\Lambda^*) = \{w \in \Lambda^* ~|~ |w| \leq r\}$ denote the ball of radius $r$ and we let $S_r(\Lambda^*) = \{w \in \Lambda^* ~|~ |w| = r\}$ the sphere or radius $r$ in the Cayley graph of this semigroup. For $A \subset \Lambda^*$ we can define the upper density $$\overline{d}(A) =\limsup_{N \to \infty} \frac{1}{N} \sum_{r=0}^{N-1} \frac{|A \cap S_r(\Lambda^*)|}{|S_r(\Lambda^*)|}.$$ We say that $w' \in \Lambda^*$ is a \textit{descendent} of $w$ if there exists an $u \in \Lambda^*$ such that $w' = wu$ and if $|u|=1$ then $w'$ is also a \textit{child} of $w$. Furstenberg and Weiss showed that if $A$ has positive density then it contains arbitrarily large arithmetic subtrees, which we may define as follows.

\begin{mydef}\label{def: arithmetic subtree} An \textit{arithmetic subtree} in $A \subset \Lambda^*$ of order $r$ is a map $\phi:B_r(\Lambda^*) \to A$ such that there exists $q \in \bZ_{>0}$ such that for all $w \in B_{r-1}(\Lambda^*)$ and $\lambda \in \Lambda$ we have that \begin{enumerate}
	\item \label{q gap condition} $|\phi(w\lambda)| = |\phi(w)| + q$
	\item \label{splitting condition} The element $\phi(w\lambda)$ is a descendent of $\phi(w)\lambda$.
\end{enumerate}

We call $q$ the $\textit{gap}$ of this arithmetic subtree.

\end{mydef}
So an arithmetic tree of order $r$ is an embedding of the full finite $|\Lambda|$-ary tree of depth $r$ where each edge is mapped to a branch of the same length $q$ and at each vertex in the image we have a splitting into $|\Lambda|$ such non-overlapping branches.

\begin{thm}[Furstenberg-Weiss \cite{FWTrees}]\label{thm: Furstenberg-Weiss trees} If $A \subset \Lambda^*$ satisfies $\overline{d}(A)>0$, then for any $r \in \bZ_{>0}$ there exists an arithmetic subtree in $A$ of order $r$.

\end{thm}

We note that the original proof of Furstenberg-Weiss used Ergodic Theory and Markov Systems, though an elegant and short combinatorial proof was later found by Pach-Solymosi-Tardos \cite{PSTTrees}.

As an example, notice that if $B \subset \bZ_{\geq 0}$ has positive upper density $$\overline{d}(B) := \limsup_{N \to \infty} \frac{|B \cap [0,N)|}{N}$$ then the set $$A_B = \bigcup_{b \in B} S_b(\Lambda^*)$$ satisfies that $\overline{d}(A_B) = \overline{d}(B)>0$, thus contains arbitrarily large arithmetic subtrees by this theorem and thus by definition, $B$ contains arbitrarily large arithmetic progressions. Thus one can recover Szemer\'edi's theorem from this theorem of Furstenberg-Weiss. Alternatively, one also sees that the case $|\Lambda| = 1$ is Szemer\'edi's theorem.

\subsection{Products of trees} As we have seen, the Furstenberg-Weiss theorem extends Szemer\'edi's theorem on Arithmetic Progressions. In this paper, we will investigate whether there is a multidimensional Furstenberg-Weiss theorem that extends the multidimensional Szemer\'edi theorem in an analogous way. 

We let $\Gamma = \Lambda^* \oplus \Lambda^*$ denote the direct sum of $\Lambda^*$ with itself, thus elements in $\Gamma$ are pairs $(w_1, w_2) \in \Lambda \times \Lambda$ and multiplication is performed ordinate-wise. We define the level of $\gamma = (w_1,w_2) \in \Gamma$, denoted by $\operatorname{Level}(\gamma) \in \bZ_{\geq 0}^2$, to be the pair $(|w_1|, |w_2|)$ were $|w|$ denotes the word length of $w \in \Lambda^*$. For $(i,j) \in \bZ_{\geq 0}^2$ we define the level set $$L_{i,j} = \{ \gamma \in \Gamma ~|~ \operatorname{Level}(\gamma) = (i,j) \}.$$

\begin{mydef} For $A \subset \Gamma$ we define the upper density $$\overline{d}(A) = \limsup_{N \to \infty} \frac{1}{N^2} \sum_{i,j=0}^{N-1} \frac{|A \cap L_{i,j}|}{|L_{i,j}|}.$$

\end{mydef}

Recall that the $2$-dimensional Szemer\'edi theorem can be formulated as saying that a set of positive density in $\bN^2$ contains a cartesian product of two arbitrarily long arithmetic progressions in $\bN$ of the same common difference. Thus it now makes sense to ask the following question.

\begin{question}\label{question: cartesian product trees} If $A \subset \Gamma$ with $\overline{d}(A)>0$ then does $A$ contain the cartesian product of two arbitrarily large arithmetic trees with the same gap? More precisely, if $r>0$ then does there exist a $q$ and two arithmetic trees $\phi_1,\phi_2:B_r(\Lambda^*) \to \Lambda^*$ with gap $q$ such that $(\phi_1(w_1), \phi_2(w_2)) \in A$ for all $w_1,w_2 \in B_r(\Lambda^*)$? \end{question}

\begin{remark}\label{remark: basic example} An affirmative answer immediately implies the two dimensional Szemer\'edi theorem as follows. If $B \subset \bN^2$ has positive upper density $$\overline{d}(B) = \limsup_{N \to \infty} \frac{|B \cap [0,N)^2|}{N^2}>0$$ then letting $A_B = \bigcup_{(i,j) \in B} L_{i,j}$ we see that $\overline{d}(A_B) = \overline{d}(B)>0$ and the existence of product of arithmetic trees of gap $q$ and order $r$ in $A_B$ implies the existence of the cartesian product of two arithmetic progessions of length $r$ and gap $q$ in $B$. 

\end{remark}

We remark that a different multidimendsional extension of the Furstenberg-Weiss theorem has been posed by Dodos and Kanellopoulos (Problem 6 in \cite{DKRamseyProduct}), though it instead considers positive density subsets of the \textit{level product} $\bigcup_{i \in \bN} L_{i,i}$ rather than the cartesian product of two trees.

We are unable to answer Question~\ref{question: cartesian product trees}, but we will turn to proving weaker extensions of the Furstenberg-Weiss theorem that still imply the multidimensional Szemer\'edi theorem in the same way as in Remark~\ref{remark: basic example}

\subsection{Arrays of horizontal trees} We state our first partial result concerning Question~\ref{question: cartesian product trees}. Note that an affirmative answer to this question would imply that for all integers $r>0$ there is some integer $q>0$ such that our positive density subset $A \subset \Gamma$ would contain $r$ horizontally embedded copies of an arithmetic subtree with gap $q$ and depth $r$ such that the vertical distance between these consecutive horizontal trees is $q$. Our first partial result says that this is indeed the case.

\begin{thm}\label{thm: weak approximation to cartesian product} For $A \subset \Gamma$ with $\overline{d}(A)>0$ and $r \in \bZ_{>0}$ there exists a positive integer $q$ and an injective map $\phi:B_r(\Lambda^*) \times \{0,1, \ldots r\} \to A$ such that the following hold.

\begin{enumerate}
	\item For some constants $c_1,c_2 \in \bZ_{\geq 0}$ we have $$\operatorname{Level}(\phi(u,j)) = (q|u| + c_1, qj + c_2) \quad\text{for all } u \in B_r(\Lambda^*), j \in \{0,1, \ldots, r\}.$$
	\item For each $j \in \{0,1 \ldots, r\}$, we have an arithmetic subtree $\phi_j: B_r(\Lambda^*) \to \Lambda^*$ with gap $q$ and $y_j \in \Lambda^*$ such that for all $u \in B_r(\Lambda^*)$, we have $\phi(u,j) = (\phi_j(u), y_j)$.
\end{enumerate}

\end{thm}

Observe that this result extends the two-dimensional Szemer\'edi theorem in $\bN^2$ by the same argument as given in Remark~\ref{remark: basic example}. The embedding provided by this result is much weaker than the sought cartesian product structure and it is natural to ask whether one can find stronger embeddings that more closely resemble a cartesian product such as ones where the $\phi_j$ are all equal.

\subsection{Free Products}
\label{section: Free Products}

We now turn to presenting our next multidimensional extension of Furstenberg-Weiss theorem that extends naturally the multidimensional Szemer\'edi theorem as in Remark~\ref{remark: basic example}. This will involve finding certain non-degenerate arithmetic embeddings of finite subsets of the free product $\Lambda^* \ast \Lambda^*$ into positive density subsets of $\Gamma = \Lambda \oplus \Lambda$.

If one views $\Gamma$ as the vertex set of the cartesian product of two trees, then each vertex $(w_1,w_2) \in \Gamma$ has horizontal children $(w_1\lambda, w_2)$ and vertical children $(w_1, w_2\lambda)$. Thus this (undirected) graph (the cartesian product of two trees) has cycles such as squares. We now pass to the universal cover to obtain an acyclic graph by considering the free product of $\Lambda^*$ as follows.

Recall that we consider $\Lambda^*$ to be the free semi-group with identity $\emptyset$ generated by $\Lambda$. The free product $\widetilde{\Gamma} = \Lambda^* \ast \Lambda^*$ is the free semigroup generated by $$\{ X_{\lambda} ~|~ \lambda \in \Lambda\} \sqcup \{Y_{\lambda} ~|~ \lambda \in \Lambda \}.$$ Thus every element in $\widetilde{\Gamma}$ can uniquely be written as a word (product) of elements from this set. We define the level of $\gamma \in \Gamma$, denoted by $\operatorname{Level}(\gamma) \in \bZ_{\geq 0}^2$, to be the pair $(n,m)$ where $n$ (resp. $m$) is the number of letters of the form $X_{\lambda}$ (resp. $Y_{\lambda}$) in the word $\gamma$. Let $B_r(\widetilde{\Gamma})$ denote the set of words in $\widetilde{\Gamma}$ of word length at most $r$. 
Recall that $\Gamma = \Lambda^* \oplus \Lambda^*$ denotes the direct sum of $\Lambda^*$ with itself, thus elements in $\Gamma$ are pairs $(w_1, w_2) \in \Lambda \times \Lambda$ and multiplication is performed ordinate-wise. Thus the natural homomorphism $\widetilde{\Gamma} \to \Gamma$ preserves levels. For $\gamma = (w_1,w_2) \in \Gamma$, we define $\gamma X_{\lambda} = (w_1 \lambda, w_2)$ and $\gamma Y_{\lambda} = (w_1, w_2 \lambda)$. In other words, we have a right action of $\Gamma \curvearrowleft \widetilde{\Gamma}$ where $\gamma \cdot \widetilde{\gamma}$ is the product of $\gamma \in \Gamma$ and the image of $\widetilde{\gamma} \in \widetilde{\Gamma}$ in $\Gamma$.

In a semigroup $S$, we say that $s' \in S$ is a \textit{descendent} of $s \in S$ if $s' = st$ for some $t \in S$.

Given $A \subset \Gamma$ and $u,v \in \bZ^2_{\geq 0}$, we define an \textit{$(u,v)$-arithmetic product tree of order $r$ in $A$} to be a map $\phi:B_r(\widetilde{\Gamma}) \to A$ such that

\begin{enumerate}
	\item For $\gamma \in B_{r-1}(\widetilde{\Gamma})$ and $\lambda \in \Lambda$  we have $$\operatorname{Level}(\phi(\gamma X_{\lambda})) = \operatorname{Level}(\phi(\gamma)) + u$$ and $$\operatorname{Level}(\phi(\gamma Y_{\lambda})) = \operatorname{Level}(\phi(\gamma)) + v.$$
	\item For $\gamma \in B_{r-1}(\widetilde{\Gamma})$ and $\lambda \in \Lambda$ we have that $\phi(\gamma X_{\lambda})$ is a descendent of $\phi(\gamma)X_{\lambda}$ and $\phi(\gamma Y_{\lambda})$ is a descendent of $\phi(\gamma)Y_{\lambda}$.
\end{enumerate}

Our main theorem may now be stated as follows.

\begin{thm}\label{thm: main thm on trees in A} Let $A \subset \Gamma$ be a subset with $\overline{d}(A)>0$ and let $u,v \in \bZ^2_{> 0}$. Then for $r>0$ there exists $n \in \bZ_{>0}$ such that there is an $(nu,nv)$-arithmetic product tree of order $r$ in $A$. \end{thm}

\begin{question}\label{question: drop positivity} Can the hypothesis $u,v \in \bZ^2_{> 0}$ be weakened to $u,v \in \bZ_{\geq 0}$? In particular, can we take $u=(1,0)$ and $v=(0,1)$?

\end{question} 

\begin{remark} Let us now show how an affirmative answer to our original Question~\ref{question: cartesian product trees} implies an affirmative answer to this question. Finding a cartesian product of two arithmetic product trees (in $\Lambda^*$) of order $r$ and gap $q$ inside $A$ means that we have a mapping $\phi' = B_r(\Lambda^*) \times B_r(\Lambda^*) \to A$ such that $\phi'(x,y) = (\phi_1(x), \phi_2(x))$ for some $\phi_1,\phi_2:  B_r(\Lambda^*) \to \Lambda^*$ that are arithmetic trees with gap $q$ and of order $r$. But then if $\pi:B_r(\widetilde{\Gamma}) \to B_r(\Lambda^*) \times B_r(\Lambda^*) $ denotes the natural homomorphism $\widetilde{\Gamma} \to \Gamma$ restricted to $B_r(\widetilde{\Gamma})$, then $\phi = \phi' \circ \pi$ is in fact a $(qu,qv)$-arithmetic product tree of order $r$ for $u = (1,0)$ and $v=(0,1)$. So indeed an affirmative answer to Question~\ref{question: cartesian product trees} implies an affirmative answer to Question~\ref{question: drop positivity} for this $u$ and $v$, but in fact it is not hard to extend this argument to show that it would answer Question~\ref{question: drop positivity} completely. In other words, Question~\ref{question: cartesian product trees} asks whether the $(u,v)$-arithmetic product trees can be chosen so that they factor through the natural homomorphism $\widetilde{\Gamma} \to \Gamma$. 

\end{remark}

\begin{remark} It is not hard to see that Theorem~\ref{thm: main thm on trees in A} implies the multidimensional Szemer\'edi theorem as follows. Let $B \subset \bZ_{\geq 0}^2$ be a set of positive density and construct $A_B \subset \Gamma$ as in Remark~\ref{remark: basic example} so that $\overline{d}(A_B)>0$. Choosing arbitrary $r>0$ and setting $u=(2,1), v=(1,2)$ we see, by Theorem~\ref{thm: main thm on trees in A}, that for some $q>0$ the set $A_B$ contains a $(qu,qv)$-arithmetic product tree of order $r$. This means that $B$ contains a set of the form $$ \Delta = \{ x qu + y qv + \alpha ~|~ x,y \in \bZ_{\geq 0}\text{ and } x+y \leq r \},$$ where $\alpha \in \bZ_{\geq 0}^2$ is the level of the the root of this arithmetic product tree. But it is not hard to see that this is equivalent to the two dimensional Szemer\'edi theorem as for $r$ large enough $\Delta$ contains a cartesian product of two arbitrarily large arithmetic progressions of the same gap.
\end{remark}

\subsection{Outline of the paper} In Section~\ref{section: array of trees proof} we will provide a proof of Theorem~\ref{thm: weak approximation to cartesian product}, which uses a quantitative extension of Theorem~\ref{thm: Furstenberg-Weiss trees} due to Pach-Solymosi-Tardos \cite{PSTTrees}. The remaining sections are devoted to the proof of Theorem~\ref{thm: main thm on trees in A}, which relies on extending the techniques of Furstenberg-Weiss \cite{FWTrees} on Markov systems. Our exposition will be self contained and we use some slightly different definitions that will be more convenient in our more complex situation. The main novelty that we introduce is the notion of a \textit{common endomorphic extension of two commuting Markov systems}, which naturally extend the endomorphic extensions used by Furstenberg-Weiss in \cite{FWTrees}.

\textbf{Acknowledgement:} The authors were partially supported by by the Australian Research Council grant DP210100162.

\section{Arrays of trees (Proof of Theorem~\ref{thm: weak approximation to cartesian product})}
\label{section: array of trees proof}

In this section, we let $\Lambda$ be a finite set and we will let $T_N := B_N(\Lambda^*) = \{ w \in \Lambda^* ~|~ |w| \leq N\}$ denote the $|\Lambda|$-ary tree with $N+1$ levels and $L_i = \{ x \in \Lambda^* ~|~ |x|=i\}$.
If $A \subset T_N \times T_N \subset \Lambda^* \times \Lambda^*$ then we define the density $$d_N(A) = \frac{1}{N^2} \sum_{0\leq i,j\leq N-1} \frac{|A \cap (L_i \times L_j)|}{|L_i \times L_j|} $$ which can conveniently be rewritten as follows $$d_N(A) = \frac{1}{N^2} \sum_{(x,y) \in A} |\Lambda|^{-|x|-|y|}.$$

Defining, for $y \in \Lambda^*$, the horizontal slice $A_y = \{ x \in \Lambda^* ~|~ (x,y) \in A\}$ we have that \begin{align} \label{eq: fubini density} d_N(A) = \frac{1}{N} \sum_{j=0}^{N-1} |\Lambda|^{-j} \sum_{|y|=j} d_N(A_y) \end{align} where we also use $d_N$ to denote the density on the one-dimensional tree (it will be clear from the context which density we use), so $$d_N(A_y) = \frac{1}{N}\sum_{x \in A_y} |\Lambda|^{-|x|} = \frac{1}{N}\sum_{i=0}^{N-1} \frac{|A_y \cap L_i|}{|L_i|}. $$

We will need the following notion introduced by Pach-Solymosi-Tardos in \cite{PSTTrees}, which is a relaxation of the notion of an arithmetic subtree introduced by Furstenberg-Weiss \cite{FWTrees}.

\begin{mydef} A \textit{regular embedding} $\phi:T_d \to T_N$ is an injective mapping such that 
\begin{enumerate}
	\item $\phi$ maps a level into a single level, i.e., if $w,w' \in T_d$ satisfy $|w| = |w'|$ then $|\phi(w)| = |\phi(w')|$.
	\item For $w \in T_{d-1}$ and $\lambda \in \Lambda$ we have that $\phi(w\lambda)$ is a descendent of $\phi(w)\lambda$.
\end{enumerate}

Note that if $\phi$ also satisfies that $|\phi(w\lambda)| = |\phi(w)| + q$ for all $w \in T_{d-1}$ and $\lambda \in \Lambda$ then we say that $\phi$ is \textit{arithmetic with gap $q$}, as in Definition~\ref{def: arithmetic subtree}. The following powerful result is implicit for $\Lambda = \{0,1\}$ in the work of Pach-Solymosi-Tardos \cite{PSTTrees} (it is exactly the statement regarding ``\textit{The set of levels occupied by the elements of $H_d$ in $T_n$...''} in their proof of Theorem B) and for arbitrary finite $\Lambda$ it is covered by Theorem 9.62 in \cite{DKRamseyProduct}, see also Equation (22) of \cite{DKTDense}.

\end{mydef}

\begin{thm}[Pach-Solymosi-Tardos \cite{PSTTrees}, see Theorem 9.6.2 in \cite{PSTTrees} for details] \label{thm: PST regular embedding} If $\delta>0$ then there exists an $\epsilon>0$ such that for sufficiently large $N$ we have that if $S \subset T_N$ with $d_N(S) \geq \delta$ then there exists a regular embedding $\phi:T_d \to T_N$ with $\phi(T_d) \subset S$ and $d \geq \epsilon N$.

\end{thm}

We will now use this to show the following result, which immediately implies Theorem~\ref{thm: weak approximation to cartesian product} as $d(A) = \limsup_{N \to \infty} d_N(A \cap (T_N \times T_N))$ for $A \subset \Gamma$.

\begin{thm}\label{thm: weak cartesian product finite} Fix $\delta>0$ and $r \in \bZ_{>0}$ . Then for arbitrarily large $N \in \bZ_{>0}$ we have the following: Let $A \subset T_N \times T_N$ with $d_N(A) \geq \delta$, then there exists a mapping $\phi:T_r \times \{0,1, \ldots r\} \to A$ such that the following hold.

\begin{enumerate}
	\item For some constants $c_1,c_2 \in \bZ_{\geq 0}$ we have $$\operatorname{Level}(\phi(u,j)) = (q|u| + c_1, qj + c_2) \quad\text{for all } u \in T_r, j \in \{0,1, \ldots, r\}.$$
	\item For each $j \in \{0,1 \ldots, r\}$, we have a regular embedding $\phi_j: T_r \to T_N$ which is arithmetic with gap $q$ and $y_j \in T_N$ such that for all $u \in T_r$, we have $\phi(u,j) = (\phi_j(u), y_j)$.
\end{enumerate}

\end{thm}

\begin{lemma} If $0 \leq a_1, \ldots, a_N \leq 1$ satisfy $\frac{1}{N} \sum_{i=1}^N a_i \geq \delta$ then for some $J \subset \{1,\ldots, N\}$ with $|J|\geq \frac{\delta}{2}N$ we have that $a_j \geq \frac{1}{2}\delta$ for all $j \in J$.

\end{lemma}

\begin{proof}[Proof of Theorem~\ref{thm: weak cartesian product finite}] Now fix $\delta>0$ and suppose that $A \subset T_N \times T_N$ satisfies $d_N(A) \geq \delta$. Thus applying this Lemma to the sum (\ref{eq: fubini density}) we have a set $J \subset \{1, \ldots N\}$ with $|J|\geq \frac{\delta}{2}N$ such that for all $j \in J$ we have that $$|\Lambda|^{-j} \sum_{|y|=j} d_N(A_y)  \geq \frac{\delta}{2}$$ and thus there exists a $y_j \in L_j$ with $d_N(A_{y_j})\geq \frac{\delta}{2}$. By Theorem~\ref{thm: PST regular embedding} there exists an $\epsilon = \epsilon(\delta) > 0$, such that, assuming $N$ is sufficiently large, we have a regular embedding $\phi_j: T_d \to A_{y_j}$, where $d = \lceil \epsilon N \rceil$. Say $B_j \subset A_{y_j}$ is the image of this regular embedding $\phi_j$ and we let $L(B_j) = \{ |x| ~|~ x \in B_j\}$. Now the set $\bigcup_{j} L(B_j) \times \{j\}$ is a subset of $\{1, \ldots, N\} \times \{1, \ldots, N\}$ with density at least $\frac{\delta}{2} \cdot \frac{d}{N} \geq \frac{\delta}{2} \epsilon(\delta) > 0$. So by the two dimensional Szemer\'edi theorem, if $N$ is sufficiently large then $\bigcup_{j} L(B_j) \times \{j\} \subset \{1, \ldots, N\}^2$ contains a cartesian product of two arithmetic progressions of length $r$ with the same common difference $q>0$. Thus our mapping $$\phi:T_r \times \{0,1, \ldots r\} \to \bigcup B_j \times \{j\}$$ may be constructed from a suitable restriction of the mapping $(u,j) \mapsto (\phi_j(u), j)$. \end{proof}

\section{Markov Systems}

\begin{mydef} A $\Lambda$-Markov system $(M, T, p)$ consists of a compact metric space $M$ and continuous maps $T_{\lambda}:M \to M$ and $p_{\lambda}:M \to [0,1]$ for each $\lambda \in \Lambda$ such that $$\sum_{\lambda \in \Lambda} p_{\lambda}(x) = 1 \quad \text{for all } x \in M.$$ The associated Markov Operator $P:C(M) \to C(M)$ is defined by $$(Pf)(x) = \sum_{\lambda \in \Lambda} p_{\lambda}(x)f(T_{\lambda}x).$$ We say that $(M,T,p)$ is endomorphic with respect to a continuous $S:M \to M$ if whenever $p_{\lambda}(x) >0$ then $S(T_{\lambda}x) = x$. \end{mydef}

Let us recall the notion of the Endomorphic Extension of a Markov system introduced in \cite{FWTrees}.

\begin{mydef} The \textit{Endomorphic Extension} of a $\Lambda$-Markov system $(M, T, p)$ is the $\Lambda$-Markov system $(\widetilde{M}, \widetilde{T}, \widetilde{p})$ given by $$\widetilde{M} = M^{\bZ_{\leq 0}},$$ $$\widetilde{p}_{\lambda}(x) = p_{\lambda}(x_0) \quad \text{for } x = (\ldots, x_{-1},x_0) \in M$$ and $$\widetilde{T}_{\lambda}(\ldots, x_{-2}, x_{-1}, x_0) = (\ldots ,x_{-1}, x_{0},T_{\lambda}x_0).$$

\end{mydef}

\begin{lemma} The $\textit{Endomorphic extension}$ is endomorphic with respect to the shift map $\widetilde{S}:\widetilde{M} \to \widetilde{M}$ given by $\widetilde{S}(\ldots, x_1, x_0) = (\ldots, x_2, x_1)$. \end{lemma}

We will now extend this notion by defining common Endomorphic Extension for two commuting Markov systems on the same space.

\begin{mydef} We say that a $\Lambda$-Markov system $(M,T,p)$ commutes with a $\Lambda'$-Markov system $(M,T', p')$ if $T_{\lambda}T'_{\lambda'} = T'_{\lambda'}T_{\lambda}$ for all $\lambda \in \Lambda$, $\lambda' \in \Lambda'$. \end{mydef}

Suppose that $(M,T,p)$ is a $\Lambda$-Markov system that commutes with a $\Lambda'$-Markov system $(M,T',p')$. Then their \textit{common endomorphic extension} is the space $\widetilde{M} = M^{\bZ_{\leq 0}^2}$. Let $x_{i,j} \in M$ denote the $(i,j)$ component of $x \in \widetilde{M}$. Define $\widetilde{T}_{\lambda}:\widetilde{M} \to \widetilde{M}$ by $(\widetilde{T}_{\lambda}x)_{i,j} = x_{i+1,j}$ for $i<0$ and $(\widetilde{T}_{\lambda})_{0,j} = T_{\lambda}x_{0,j}$. In other words, $\widetilde{T}_{\lambda}$ shifts horizontally to the left and inserts $T_{\lambda}(x_{0,j})$ to the vacant $0,j$ component. We define $\widetilde{S}:\widetilde{M} \to \widetilde{M}$ to be the right shift $(\widetilde{S})_{i,j} = x_{i - 1,j}$. Thus we have a $\Lambda$-Markov system $(\widetilde{M}, \widetilde{T}, \widetilde{p})$ endomorphic with respect to $\widetilde{S}$ where $\widetilde{p}_{\lambda}(x) = p_{\lambda}(x_{0,0})$. Similairly, we have a $\Lambda'$-Markov system $(\widetilde{M}, \widetilde{T'}, \widetilde{p'})$ that is endomorphic with respect to the vertical shift $\widetilde{S'}:\widetilde{M} \to \widetilde{M}$ given by $(\widetilde{S'}x)_{i,j} = x_{i, j - 1}$. Here $(\widetilde{T'}x)_{i,j} = x_{i,j+1}$ for $j<0$ and $(\widetilde{T'}x)_{i,0} = (Tx)_{i,0}$; with $\widetilde{p'}_{\lambda}(x) = p'_{\lambda}(x_{0,0})$.

\begin{lemma} For $\lambda \in \Lambda$ and $\lambda' \in \Lambda'$ we have that $\widetilde{T_{\lambda}}$ and $\widetilde{T'_{\lambda'}}$ commute.
\end{lemma}

\begin{proof} We need to consider a few cases by checking the components. First $\widetilde{T}_{\lambda}\widetilde{T'}_{\lambda'}x$ has $(0,0)$ component $T_{\lambda}T'_{\lambda'}x_{0,0}$ and $\widetilde{T'}_{\lambda'}\widetilde{T}_{\lambda}x$ has $(0,0)$ component $T'_{\lambda'}T_{\lambda}x_{0,0}$. By commutativity of the two original Markov systems $(M,T,p)$ and $(M,T',p')$, these two $(0,0)$ components coincide. The other cases are easy to check and do not require this commutativity assumption. \end{proof}

We will also need the following Lemma, which is easy to establish.

\begin{lemma} \label{lemma: S_2 and P_1 commute} For $\lambda \in \Lambda$ and $\lambda' \in \Lambda'$ we have that $\widetilde{T_{\lambda}}$ and $\widetilde{S'_{\lambda'}}$ commute.
\end{lemma}

\section{Stationary Measures}

Given a $\Lambda$-Markov system $(M,T,p)$, we say that a probability measure $\mu$ on $M$ is stationary if $P^* \mu = \mu$ where $P$ denotes the assocated Markov operator. In other words $$\int Pf d\mu = \int f d\mu \quad \text{for all } f \in C(M).$$ We say that a Markov system has \textit{constant transition probabilities} if $p_{\lambda}(x_1) = p_{\lambda}(x_2)$ for all $x_1,x_2 \in X$ and $\lambda \in \Lambda$. For simplicity of language, in this case we will treat $p_{\lambda}$ as element of $[0,1]$ rather than a function on $M$.

\begin{lemma} Suppose that $(M,T,p)$ and $(M,T',p')$ are two $\Lambda$-Markov systems that commute and that have constant transition probabilities. Then their associated Markov operators commute. 

\end{lemma}

\begin{proof} For $f \in C(M)$ and $x \in M$ we have that \begin{align*} P'Pf(x) &= P'(Pf)(x)\\ &= \sum_{\lambda' \in \Lambda} p'_{\lambda'}(x) (Pf)(T'_{\lambda'}x)\\ &=  \sum_{\lambda' \in \Lambda} p'_{\lambda'}(x) \sum_{\lambda \in \Lambda} p_{\lambda}(T'_{\lambda'}(x)) f(T_{\lambda}T'_{\lambda'}x)\\ &= \sum_{\lambda, \lambda' \in \Lambda} p'_{\lambda'}(x) p_{\lambda}(T'_{\lambda'}(x)) f(T_{\lambda}T'_{\lambda'}x)    \end{align*}
But by commutativity and the assumption that $p_{\lambda}$ and $p'_{\lambda'}$ are constant on $M$, we see that this expression is symmetric in $T$ and $T'$, hence the desired result. \end{proof}

\begin{lemma} Two Markov systems on a compact metric space $M$ whose Markov operators $P_1$ and $P_2$ commute have a common stationary measure. \end{lemma}

\begin{proof} Let $\mu$ be any measure on $M$ and define $$\mu_N = \frac{1}{N^2}\sum_{i,j=0}^{N-1} (P_{1}^*)^{i} (P_{2}^*)^{j} \mu.$$ By compactness this has a weak-$^*$ limit and one readily sees that it is invariant under both Markov operators as $$\left|\int f d\mu_N - \int Pf d\mu_N \right| = O(1/N)\|f\|_{\infty}$$ for all $f \in C(M)$ and $P=P_1,P_2$. \end{proof}

\begin{prop} \label{prop: common stationary measure} Let $\widetilde{M}$ be the common endomorphic extension of two commuting $\Lambda$-Markov systems $(M,T_1,p_1)$ and $(M, T_2,p_2)$ with constant transition probabilities and let $\mu$ be a measure for $M$ that is stationary for both systems. Then there exists a measure $\widetilde{\mu}$ on $\widetilde{M}$ that is stationary for $(\widetilde{M}, \widetilde{T_1}, \widetilde{p_1})$ and $(\widetilde{M}, \widetilde{T_2}, \widetilde{p_2})$ and such that $\pi^* \widetilde{\mu} = \mu$ where $\pi:\widetilde{M} \to M$ is the factor map $\pi(x) = x_{0,0}$ for $x \in \widetilde{M} = M^{\bZ^2_{\leq 0}}$. \end{prop}

\begin{proof} Let $\widetilde{\mu}_0$ be any measure (not necessarily stationary) on $\widetilde{M}$ such that $\pi^* \widetilde{\mu}_0 = \mu$ (for example, one can take $\widetilde{\mu}_0$ to be the product measure $\prod_{i,j=0}^{\infty} \mu$). Now let $$\widetilde{\mu}_N = \frac{1}{N^2}\sum_{i,j=0}^{N-1} (\widetilde{P}_{1}^*)^{i} (\widetilde{P}_{2}^*)^{j} \widetilde{\mu}_0.$$ Observe that for $i=1,2$ we have that $$\pi \circ (\widetilde{T_i})_{\lambda} = (T_i)_{\lambda} \circ \pi$$ and that $(p_i)_{\lambda}$ and $(\widetilde{p}_i)_{\lambda}$ are equal (more precisely, they are constant functions, on $M$ and $\widetilde{M}$ respectively, with the same value) and so we must have that $P_i^* \pi^* = \pi^* \widetilde{P_i}^*$. This relation now implies that $\pi^* \widetilde{\mu}_N = \mu$. Thus if we take a weak$^*$-limit of $\widetilde{\mu}_N$ then it is a stationary measure that also projects to $\mu$ on $M$. \end{proof}

We now state some basic facts about stationary measures on endomorphic systems already established in \cite{FWTrees}. We give proofs for completeness.

\begin{prop}\label{prop: P S basics} Let $\mu$ be a stationary measure on a $\Lambda$-Markov system $(M,T,p)$ that is endomorphic with respect to a continuous $S:M \to M$. Let $P:C(M) \to C(M)$ be the corresponding Markov operator on $M$. Then

\begin{enumerate}
	\item The operator $P$ extends uniquely to $P:L^q(M,\mu) \to L^q(M,\mu)$ for $1\leq q \leq \infty$, which has operator norm $1$ and is defined by the same formula $$Pf(x) = \sum_{\lambda \in \Lambda} p_{\lambda}(x)f(Tx).$$
	\item For any $h,g \in L^{\infty}(M,\mu)$ we have (where $Sg$ is the function $Sg(x) :=g(Sx)$) $$P(h\cdot Sg) =  (Ph) \cdot g.$$
	\item $P$ is a left inverse of $S$. That is, $PS=\operatorname{I}$.
	\item The map $S$ preserves $\mu$.
	\item $P$ and $S$ are adjoint operators on $L^2(M,\mu)$.
\end{enumerate}

\end{prop}

\begin{proof} Let $f:M \to \bC$ be measurable. Then for $q\geq 1$ we have,
\begin{align*} |Pf(x)| &= \left| \sum_{\lambda \in \Lambda} p_{\lambda}(x) f(T_{\lambda}x) \right| 
\\ &\leq \sum_{\lambda \in \Lambda} p_{\lambda}(x) |f(T_{\lambda}x)| 
\\ &\leq  \left( \sum_{\lambda \in \Lambda} p_{\lambda}(x) |f(T_{\lambda}x)|^q \right)^{1/q}
\\ &= \left(P(|f|^q)(x)\right)^{1/q} \end{align*}

Actually, the first inequality shows that $\| Pf \|_{\infty} \leq \|f \|_{\infty}$ and so $P$ has operator norm $1$ on $L^{\infty}$. Now for $1\leq q<\infty$, if we integrate the estimate $|Pf|^q \leq P(|f|^q)$ just established we get $$\int |Pf|^q d\mu \leq \int P(|f|^q) d\mu = \int |f|^q d\mu.$$ This shows that $P$ is an operator on $L^{q}(M,\mu)$ with operator norm $1$. Now if $h,g \in L^{\infty}(M,\mu)$ then 
\begin{align*} P(h \cdot Sg)(x) &= \sum_{\lambda \in \Lambda} p_{\lambda}(x) (h \cdot Sg) (T_{\lambda}x)
\\ &= \sum_{\lambda \in \Lambda} p_{\lambda}(x) h(T_{\lambda}x) g(ST_{\lambda}x)
\\ &= \sum_{\lambda \in \Lambda} p_{\lambda}(x) h(T_{\lambda}x) g(x)
\\ &= Ph(x) \cdot g(x). \end{align*}

In particular, if we take $h=1$ and use $Ph=1$, then we get the identity $P(Sg) = g$, so $PS=I$. Integrating this we get $$\int g d\mu = \int P(Sg) d\mu = \int Sg d\mu,$$ so $S$ preserves $\mu$. Finally, $$\langle h, Sg \rangle = \int h \cdot Sg d\mu = \int P(h \cdot Sg) d\mu = \int Ph \cdot g d\mu = \langle Ph, g \rangle$$ thus $P$ and $S$ are indeed adjoint on $L^2(M,\mu)$.\end{proof}

\section{Trees in Markov Systems}

Given a $\Lambda$-Markov system $(M, T, p)$ a \textit{path from $x$ to $x'$ in $(M,T,p)$} is a sequence of points $x_0, \ldots, x_n$ such that $x_0 = x$, $x_n = x'$ and for each $i \in \{0, \ldots, n-1\}$ there exists $\lambda_i \in \Lambda$ such that $x_{i+1} = T_{\lambda_i} x_i$ and $p_{\lambda_i}(x_i) > 0$. Given such a path, we say that its \textit{length} is $n$ and we say that it has \textit{initial direction} $\lambda_0$ if $n>0$. 

\begin{mydef} \label{def: tree in Markov system} Suppose that $\mathcal{M}_1 = (M,T_1,p_1)$ and $\mathcal{M}_2 = (M,T_2,p_2)$ are commuting $\Lambda$-Markov systems and let $A \subset M$. Given vectors $u = (u_1,u_2),v = (v_1,v_2) \in \bZ_{\geq 0}^2$ we define an $(u,v)$-arithmetic product tree of order $r$ in $M$ to be a map $\phi:B_r(\widetilde{\Gamma}) \to A$ such that for $\gamma \in B_{r-1}(\widetilde{\Gamma})$ and $\lambda \in \Lambda$ we have that:

\begin{enumerate}
	\item  There is a path in $(M,T_1,p_1)$ of length $u_1$ from $\phi(\gamma)$ to some point $x \in M$ with initial direction $\lambda$ followed by a path of length $u_2$ in $(M,T_2,p_2)$ from this point $x$ to $\phi(\gamma X_{\lambda})$.
	\item There is a path in $(M,T_2,p_2)$ of length $v_2$ from $\phi(\gamma)$ to some point $x \in M$ with initial direction $\lambda$ followed by a path of length $v_1$ in $(M,T_1,p_1)$ from this point $x$ to $\phi(\gamma Y_{\lambda})$
\end{enumerate}

We say that $\phi(\emptyset)$ is the \textit{root} of this $(u,v)$-arithmetic product tree.

\end{mydef}

The main result concerns the existence of embeddings of arithmetic product trees in such commuting Markov systems.

\begin{mydef} We say that a $\Lambda$-Markov system $(M,T,p)$ is non-degenerate if for all $\lambda \in \Lambda$, the function $p_{\lambda}$ is not $0$ on all of $M$. Moreover, we say that it has disjoint images if $T_{\lambda}(M) \cap T_{\lambda'}(M) = \emptyset$ for distinct $\lambda, \lambda' \in \Lambda$.\end{mydef}

\begin{thm}\label{thm: arithmetic product trees in Markov} Let  $\mathcal{M}_1 = (M,T_1,p_1)$ and $\mathcal{M}_2 = (M,T_2,p_2)$ be non-degenerate commuting $\Lambda$-Markov systems with constant transition probabilities and disjoint images. Suppose that $\mu$ is a common stationary measure for both of these systems and $A \subset M$ a measurable set with $\mu(A)>0$. Then for each positive integer $r>0$ and $u,v \in \bZ^2_{>0}$ there exists a positive integer $n$ such that there exists a $(nu,nv)$-arithmetic product tree of order $r$ in $A$. Moreover, there is a positive measure subset $A' \subset A$ such that all $a \in A'$ are roots of a $(nu,nv)$-arithmetic product tree of order $r$ in $A$. \end{thm}

We first show that we may reduce this to the case of Endomorphic systems as follows.

\begin{thm}\label{thm: arithmetic product trees in Endomorphic}  Let  $\mathcal{M}_1 = (M,T_1,p_1)$ and $\mathcal{M}_2 = (M,T_2,p_2)$ be commuting $\Lambda$-Markov systems with constant transition probabilities and disjoint images that are endomorphic with respect to continous $S_1:M \to M$ and $S_2:M \to M$ respectively. Suppose also that $(T_2)_{\lambda}$ commutes with $S_1$ for all $\lambda \in \Lambda$. Suppose that $\mu$ is a common stationary measure for both of these systems and $A \subset M$ a measurable set with $\mu(A)>0$. Fix $u,v \in \bZ^2_{>0}$. Then for each positive integer $r>0$ there exists a positive integer $n$ such that there exists a $(nu,nv)$-arithmetic product tree of order $r$ in $M$. Moreover, there is a positive measure subset $A' \subset A$ such that all $a \in A'$ are roots of a $(nu,nv)$-arithmetic product tree of order $r$. \end{thm}

\begin{proof}[Proof of Theorem~\ref{thm: arithmetic product trees in Markov} from Theorem~\ref{thm: arithmetic product trees in Endomorphic}] We let $\widetilde{M}$ be the common endomorphic extension of these systems and let $\pi:\widetilde{M} \to M$ be the natural map as in Proposition~\ref{prop: common stationary measure}. By Proposition~\ref{prop: common stationary measure}, we may find a measure $\widetilde{\mu}$ that is stationary for both of the systems on $\widetilde{M}$ such that $\pi^* \widetilde{\mu} = \mu$. Let $\widetilde{A} = \pi^{-1}(A)$, thus $\widetilde{\mu}(\widetilde{A}) = \mu(A) >0.$ By Lemma~\ref{lemma: S_2 and P_1 commute}, we may apply Theorem~\ref{thm: arithmetic product trees in Endomorphic} to find an $n$ such that an $(nu,nv)$-arithmetic product tree $\phi:B_r(\widetilde{\Gamma}) \to \widetilde{A}$ exists. Finally $\pi \circ \phi$ is an arithmetic product tree in $A$, which can be justified by the identities $\pi ((\widetilde{T}_i)_{\lambda}x) = (T_i)_{\lambda} (\pi(x))$ and $p_{\lambda}(\pi(x)) = \widetilde{p}_{\lambda}(x)$, for $x \in \widetilde{M}$, $\lambda \in \Lambda$ and $i =1,2$, that show that $\pi$ preserves paths of any given length and initial direction. Finally, the set $A' \subset A$ of roots of such trees is measurable and so if $\mu(A') = 0$ then $\widetilde{\mu}(\pi^{-1}A') = 0$, which would contradict the positivity of the measure of the set of roots of such trees in $\widetilde{A}$, thus $\mu(A')>0$. \end{proof}

\section{Recurrence in Endomorphic Systems}

We prove in this section Theorem~\ref{thm: arithmetic product trees in Endomorphic}. So fix two non-degenerate $\Lambda$-Markov systems $(M,T_1,p_1)$ and $(M,T_2,p_2)$ that commute and have constant transition probabilities and are Endomorphic with respect to two commuting continuous maps $S_1:M \to M$ and $S_2:M \to M$ respectively. Thus their corresponding Markov Operators $P_1$ and $P_2$ commute. Recall that we also assume that $S_2$ and $(T_1)_{\lambda}$ commute for all $\lambda \in \Lambda$ (cf. Lemma~\ref{lemma: S_2 and P_1 commute}) and thus $S_2$ and $P_1$ commute. Let $\mu$ be a measure on $M$ that is stationary for both processes. Thus $S_1$ and $S_2$ preserve $\mu$ (by Proposition~\ref{prop: P S basics}). We now let $S_1,S_2$ also denote the corresponding Koopman operators on $L^2(M,\mathcal{B}, \mu)$. Note that (by Proposition~\ref{prop: P S basics}) $P_1,P_2$ extend to norm $1$ operators on $L^2(M,\mathcal{B}, \mu)$ and $L^{\infty}(M,\mu)$ and that $P_i$ is a left inverse of $S_i$, that is $$P_iS_i = \operatorname{I} \quad\text{for } i =1,2.$$ Now let $$H_{n,m} = L^2(M, S_1^{-n}S_2^{-m}\mathcal{B}, \mu)$$ denote the space of $L^2$-functions measurable with respect to the sub-$\sigma$-algebra  $S_1^{-n}S_2^{-m}\mathcal{B}$, which we note is precisely the set of those functions written as $$S_1^{n}S_2^{m}g \quad\text{for some } g \in L^2(M,\mathcal{B}, \mu).$$ Now let $$H_{\infty} = \bigcap_{n,m\geq 0} H_{n,m} = L^2(M, \bigcap_{n,m \geq 0} S_1^{-n}S_2^{-m}\mathcal{B}, \mu)$$ and let $Q_{n,m}$ and $Q_{\infty}$ denote the orthogonal projection onto $H_{n,m}$ and $H_{\infty}$ respectively. Note that for all $f \in L^2(M,\mathcal{B}, \mu)$ and sequences $n_i \to \infty, m_i \to \infty$ we have that $$\lim_{i \to \infty} Q_{n_i, m_i}f = Q_{\infty}f.$$ Note also that $Q_{n,m}$ and $Q_{\infty}$ are conditional expectations and hence have $L^{\infty}$ norm at most $1$.

\begin{lemma} We have that $S_1^{n}P_1^{n}S_2^{m}P_2^{m} = Q_{n,m}$.

\end{lemma}

\begin{proof} First we show that if $f \in H_{n,m}$ then the left hand side fixes $f$. By definition, this means that $f = S_1^n S_2^m g$ for some $g$. Thus we have $$ S_1^{n}P_1^{n}S_2^{m}P_2^{m}f = S_1^{n}P_1^{n}S_2^{m}P_2^{m} S_1^n S_2^m g = S_1^n S_2^m g = f$$ where we have used the commutativity of $S_1$ and $S_2$ together with the relations $P_iS_i = I$, thus showing that the left hand side is indeed the identity on $H_{n,m}$. Now we show that if $f$ is orthogonal to $H_{n,m}$ then so is the left hand side applied to $f$. If $h \in H_{n,m}$ then $$\langle S_1^{n}P_1^{n}S_2^{m}P_2^{m} f, h \rangle = \langle f, (S_1^{n}P_1^{n}S_2^{m}P_2^{m})^* h \rangle = \langle f, S_2^m P_2^m S_1^n P_1^n h \rangle = \langle f, S_2^m  S_1^n P_2^m P_1^n h \rangle$$ where we have used the fact that $S_i$ and $P_i$ are adjoint (see Proposition~\ref{prop: P S basics}) followed by the commutativity of $S_1$ and $P_2$ (which follows from the fact that $S_2 = P_2^*$ and $P_1 = S_1^*$ commute). However, we see that $ S_2^m  S_1^n P_2^m P_1^n h \in H_{n,m}$ and so the fact that $f$ is orthogonal to $H_{n,m}$ implies that this inner product is indeed $0$.   \end{proof}

\begin{lemma} We have that $H_{\infty}$ is $P_1$,$P_2$, $S_1$ and $S_2$ invariant. Moreover, $P_i$ and $S_i$ are inverses on $H_{\infty}$. \end{lemma}

\begin{proof} If $f \in H_{n,m}$, with $n,m>0$, then $f = S_1^n S_2^m g$ for some $g$. Hence using $P_1S_1 = \operatorname{I}$ we get that $P_1f = S_1^{n-1}S_2^m g \in H_{n-1,m}.$ So $P_1$ maps $H_{n,m}$ to $H_{n-1,m}$ and similarly, by commutativity of $S_1$ and $S_2$, we can show that $P_2$ maps $H_{n,m}$ to $H_{m,n-1}$. As $H_{n,m}$ is a subset of $H_{n-1,m}$ and $H_{n,m-1}$, we can see that $H_{\infty}$ is indeed invariant under $P_1$ and $P_2$. Clearly $S_1$, $S_2$ preserve $H_{n,m}$ thus $H_{\infty}$ as well. Now given $f \in H_{\infty}$ we have that $f = S_1g$ for some $g$ and hence $$S_1P_1f = S_1P_1S_1g = S_1g = f,$$ showing that $S_1$ is a left inverse of $P_1$. \end{proof}

We can now see that for any $f \in L^2(M, \mathcal{B},\mu)$ we have that \begin{align} \label{Q infinity limit} \| P_1^n P_2^m f - S_1^{-n} S_2^{-m} Q_{\infty} f \|_2 = \|S_1^n S_2^m P_1^n P_2^m f - Q_{\infty} f\|_2 = \| Q_{n,m}f - Q_{\infty}f\| \to 0 \quad\text{as } n,m \to \infty. \end{align}

Let $u_1,u_2, v_1,v_2 \in \bZ_{>0}$ and let $$W_1 = P_1^{u_1}P_2^{u_2}, W_2 = P_1^{v_1}P_2^{v_2}.$$ Now for $\lambda \in \Lambda$ we let $B_{\lambda} = (T_1)_{\lambda}(M)$ and $C_{\lambda} = (T_2)_{\lambda}(M)$. Note that the $B_{\lambda}$ are disjoint by assumption, likewise for $C_{\lambda}$. Moreover, by disjointness we have that $P_1(\mathds{1}_{B_{\lambda}}h)(x) = (p_1)_{\lambda}h((T_1)_{\lambda}x)$ for any $h \in L^{\infty}(M,\mu)$ and likewise for $C_{\lambda}$, which we will use throughout. In particular, this means that $P_1(\mathds{1}_{B_{\lambda}}) = (p_1)_{\lambda}>0$ is a positive constant function and so the function $$\varphi_1 = \mathds{1}_A \prod_{\lambda \in \Lambda} (P_1 \mathds{1}_{B_{\lambda}} ) \cdot (P_2 \mathds{1}_{C_{\lambda}}) $$ has positive integral. Note that if $x \in A$ is such that $\varphi_1(x)>0$ then that means that for each $\lambda$ with positive probability $(p_1)_{\lambda}(x) >0$ one can land in $B_{\lambda}$, and likewise for $C_{\lambda}$. Now let $$\varphi_2 = \mathds{1}_A \prod_{\lambda \in \Lambda} P_1(\mathds{1}_{B_{\lambda}} \cdot P_1^{u_1n-1}P_2^{u_2n}\varphi_1) \cdot   P_2(\mathds{1}_{C_{\lambda}} \cdot P_1^{v_1 n}P_2^{v_2n-1}\varphi_1).
$$ Note that if $x \in M$ is such that $\varphi_2(x)>0$ then that means that $x \in A$ and one may find, for each $\lambda \in \Lambda$, a path $x_0, x_1,\ldots, x_{u_1n}$ in the Markov system $(M,T_1,p_1)$ of length $n$ which starts at $x_0 = x$, then immediately hits $B_{\lambda}$ (that is, $x_1 \in B_{\lambda}$) and so this path has initial direction $\lambda$ (by disjointness of the $B_{\lambda}$), followed by a path $y_0, \ldots, y_{u_2 n}$ in the Markov system $(M,T_2,p_2)$ with $y_0 = x_n$ and $\varphi_1(y_{u_2 n}) >0$, thus $y_{u_2n} \in A$. Likewise, the same implication holds for $C_{\lambda}$ and $(v_1,v_2)$ in place of $B_{\lambda}$ and $(u_1,u_2)$. Hence, using the language of Definition~\ref{def: tree in Markov system} we have that $x$ is the root of some $(nu,nv)$-arithmetic tree of order $1$ in $A$, where $u = (u_1,u_2)$ and $v = (v_1, v_2)$. From (\ref{Q infinity limit}), we now have for large enough $n$ that $$P_1^{u_1n - 1} P_2^{u_2n} \varphi_1 = S_1^{-u_1n + 1} S_2^{-u_2n}\varphi_{\infty} + o(1)$$ where $\varphi_{\infty} = Q_{\infty}\varphi_1$ and likewise for $v_1,v_2$ in place of $u_1,u_2$. Thus for large enough $n$ we have that $\varphi_2$ is approximately (in $L^2$) 

$$\psi_2 =  \mathds{1}_A \prod_{\lambda \in \Lambda} P_1(\mathds{1}_{B_{\lambda}} \cdot S_1^{-u_1n + 1} S_2^{-u_2n}\varphi_{\infty}) \cdot   P_2(\mathds{1}_{C_{\lambda}} \cdot S_1^{-v_1 n} S_2^{-v_2 n + 1}\varphi_{\infty}),$$ that is, since $\psi_2$ is almost $\varphi_2$ for $n$ large enough, it will be enough to show that $\psi_2$ has positive integral for arbitrarily large $n$ in order to verify the same for $\varphi_2$.

Using the identity $P(h \cdot Sg) = Ph \cdot g$, for $P = P_1,P_2$, we obtain that 
\begin{align*} \psi_2 &=  \left(\mathds{1}_A \prod_{\lambda \in \Lambda} (P_1 \mathds{1}_{B_{\lambda}} )\cdot (P_2 \mathds{1}_{C_{\lambda}}) \right)  \left(S_1^{-u_1n} S_2^{-u_2n}\varphi_{\infty}\right)^{|\Lambda|} \left(S_1^{-v_1n} S_2^{-v_2n}\varphi_{\infty}\right)^{|\Lambda|} \\ &= \varphi_1 \left(S_1^{-u_1n} S_2^{-u_2n}\varphi_{\infty}\right)^{|\Lambda|} \left(S_1^{-v_1n} S_2^{-v_2n}\varphi_{\infty}\right)^{|\Lambda|}\\ &= \varphi_1 \cdot W_1^{n} \varphi_{\infty}^{|\Lambda|} \cdot W_2^{n} \varphi_{\infty}^{|\Lambda|} \end{align*} 
where in the final step we have used the fact that $W_1$ and $W_2$ are inverses of Koopman operators and thus are multiplicative (on $L^{\infty}$). Since we wish to show that $\psi_2$ has a positive integral (for arbitrarily large $n$), it will be enough to show that $Q_{\infty} \psi_2$ has a positive integral as $Q_{\infty}$ is a conditional expectation operator. Moreover, the fact that it is a conditional expectation operator means that $Q_{\infty}(f \cdot g) = (Q_{\infty}f) \cdot g$ for $g$ in $H_{\infty}$ and $f,g$ in $L^{\infty}$, hence 
$$Q_{\infty} \psi_2 =  \varphi_{\infty} \cdot W_1^{n} \varphi_{\infty}^{|\Lambda|} \cdot W_2^{n} \varphi_{\infty}^{|\Lambda|}.$$ 
Now since $\varphi_{\infty}$ and $\varphi_{\infty}^{|\Lambda|}$ are positive functions with the same support, the Ergodic Szemer\'edi theorem for commuting transformations \cite{FKSZ} (applied to $W_1$ and $W_2$) demonstrates that for arbitrarily large $n$, the integral of $Q_{\infty}\psi_2$ is bounded away from zero (greater than a positive constant depending only on $\psi_2$).

We now proceed to show the existence of order $r$ arithmetic trees in $A$ by studying higher order recurrence. To do this, we define recursively $$\varphi_r = \mathds{1}_A \prod_{\lambda \in \Lambda} P_1(\mathds{1}_{B_{\lambda}} \cdot P_1^{u_1n-1}P_2^{u_2n}\varphi_{r-1}) \cdot   P_2(\mathds{1}_{C_{\lambda}} \cdot P_1^{v_1 n}P_2^{v_2 n - 1}\varphi_{r-1}).
$$ for $r \geq 2$ (for $r=2$ this is indeed the definition of $\varphi_2$ above). One can prove by induction that $\varphi_{r}(x) > 0$ implies that $x$ is the root of an order $r-1$ $(nu,nv)$-arithmetic tree as follows. We have already argued the base case above. For the induction step, notice that $\varphi_{r}(x)>0$ implies that $x \in A$ and one may find, for each $\lambda \in \Lambda$, a path $x_0, x_1,\ldots, x_{u_1n}$ in the Markov system $(M,T_1,p_1)$ of length $n$ which starts at $x_0 = x$, then immediately hits $B_{\lambda}$ (that is, $x_1 \in B_{\lambda}$) and so this path has initial direction $\lambda$, followed by a path $y_0, \ldots, y_{u_2 n}$ in the Markov system $(M,T_2,p_2)$ with $y_0 = x_n$ and $\varphi_{r-1}(y_{u_2 n}) >0$, thus by the inductive hypothesis $y_{u_2 n}$ is the root of an order $r-2$ $(nu,nv)$-arithmetic tree in $A$. Likewise, the same implication holds for $C_{\lambda}$ and $(v_1,v_2)$ in place of $B_{\lambda}$ and $(u_1,u_2)$. Using the recursive structure of $(nu,nv)$-arithmetic trees we get that indeed $x$ is the root of an $(nu,nv)$-arithmetic tree of order $(r-2)+1=r-1$, completing the inductive step. As shown for $\varphi_1$, we can show that $$\| P_1^{u_1n - 1} P_2^{u_2n} \varphi_{r-1} - S_1^{-u_1n + 1} S_2^{-u_2n}Q_{\infty}\varphi_{r-1} \|_2 \to 0 \quad \text{as } n \to \infty.$$ Hence we have that $\| \psi_r - \varphi_r \|_2 \to 0$ as $n \to \infty$ where $$\psi_r =  \mathds{1}_A \prod_{\lambda \in \Lambda} P_1(\mathds{1}_{B_{\lambda}} \cdot S_1^{-u_1n + 1} S_2^{-u_2n}Q_{\infty}\varphi_{r-1}) \cdot   P_2(\mathds{1}_{C_{\lambda}} \cdot S_1^{-v_1 n} S_2^{-v_2 n + 1}Q_{\infty}\varphi_{r-1})$$ which, by the same calculations as above, we can write as $$\psi_r = \varphi_1 \cdot  \left(W_1^{n} Q_{\infty}\varphi_{r-1}\right)^{|\Lambda|} \cdot  \left(W_2^{n} Q_{\infty}\varphi_{r-1}\right)^{|\Lambda|}.$$ Now as $\varphi_{r-1}$ is close to $\psi_{r-1}$ for large $n$ and $\|\psi_r\|_{\infty} \leq 1$, we see that $\psi_r$ is $L^2$-close to $$\varphi_1 \cdot \left(W_1^{n} Q_{\infty}\psi_{r-1}\right)^{|\Lambda|} \cdot \left(W_2^{n} Q_{\infty}\psi_{r-1}\right)^{|\Lambda|}.$$ Now by induction on $r$ we see that, for large $n$, we have that $\psi_r$ is $L^2$ close to a product of the form $$\psi'_r = \varphi_1 \prod_{a,b} W_1^{an}W_2^{bn}\varphi_{\infty}$$ where the product is taken over some finite multiset of non-negative integers $a,b$. As before, we see that $$Q_{\infty}\psi'_r = \varphi_{\infty} \prod_{a,b} W_1^{an}W_2^{bn}\varphi_{\infty}$$ and again by the Ergodic Szemer\'edi theorem for commuting transformations we see that this function's integral is bounded away from zero for sufficiently large $n$.

\section{Markov Processes on Labelled trees}
 
In this section we will prove Theorem~\ref{thm: main thm on trees in A} using our main Theorem~\ref{thm: arithmetic product trees in Markov} on arithmetic product trees in Markov systems. Let $\Lambda$ be a finite set and let $$\mathcal{T} = \{0,1\}^{\Lambda^* \times \Lambda^*}$$ Thus an element in $\tau \in \mathcal{T}$ is a $\{0,1\}$ vertex labelled tree, i.e., a function from the product of two trees $\tau:\Lambda^* \times \Lambda^* \to \{0,1\}$ with each vertex pair $(w_1,w_2)$ labelled $\tau(w_1,w_2)$. We now define two right actions of $\Lambda^*$ on $\mathcal{T}$ as follows. For $w \in \Lambda^*$ and $\tau \in \mathcal{T}$ let $\tau X_w \in \mathcal{T}$ denote the $\{0,1\}$ labelled tree by $$ \tau X_w (w_1,w_2) = \tau (w w_1, w_2).$$

\begin{lemma} This defines a right action, i.e., $\tau X_{\emptyset} = \tau$ and $$\tau X_wX_{w'} = \tau X_{ww'}$$. \end{lemma}

\begin{proof} For $(w_1,w_2) \in \Lambda^* \times \Lambda^* $ we have that 
\begin{align*} (\tau X_wX_{w'}) (w_1,w_2) &= ((\tau X_w)X_{w'}) (w_1,w_2) \\&= (\tau X_w) (w'w_1,w_2) \\&= \tau (ww'w_1,w_2) \\&= (\tau X_{ww'})(w_1,w_2). \end{align*} \end{proof}

Similarly, we have a right action given by defining for $w \in \Lambda^*$ and $\tau \in \mathcal{T}$ the element $\tau Y_w \in \mathcal{T}$ given by $$\tau Y_w (w_1,w_2) = (w_1,ww_2) \quad \text{for } (w_1,w_2) \in \Lambda^* \times \Lambda^*.$$ Clearly these two actions commute, i.e., $\tau X_w Y_{w'} = \tau Y_{w'} X_w$.

Now define a new space of labelled vertex-labelled trees $M = \mathcal{T} \times \Lambda \times \Lambda$, i.e., an element $\pi \in M$ is now a triple $\pi = (\tau, \lambda_1,\lambda_2)$, which we think of as $\tau \in \mathcal{T}$ labelled by a pair $(\lambda_1,\lambda_2) \in \Lambda \times \Lambda$. We define a right action of $\Lambda^*$ by defining $$(\tau, \lambda_1, \lambda_2) X_w = (\tau X_w, t(w), \lambda_2)$$ for non-empty $w = \lambda_1 \cdots \lambda_n \in \Lambda^*$ where $t(w) = \lambda_n$ denotes the final letter (if $w$ is empty, then we define $X_w$ act trivially). Analogously, we define $$(\tau, \lambda_1, \lambda_2) Y_w = (\tau Y_w, \lambda_1, t(w)).$$ Clearly, these two actions commute, i.e., $$\pi X_w Y_{w'} = \pi Y_{w'} X_w \quad \text{for } \pi \in M, w,w' \in \Lambda^*.$$

We may define a $\Lambda$-Markov System on $M$ by defining $p_{\lambda}(x) = \frac{1}{|\Lambda|}$ and $(T_1)_{\lambda}\pi = \pi X_{\lambda}$ and another $\Lambda$-Markov system (with the same transition probabilities) by defining $(T_2)_{\lambda}\pi= \pi  Y_{\lambda}$. Note that these two Markov Systems commute and hence so do their Markov Operators $P_1$ and $P_2$. For $w \in \Lambda^* \times \Lambda^*$, we define $\pi(w) := \tau(w)$ for $\pi = (\tau, \lambda_1,\lambda_2) \in M$.

Given any subset $A \subset \Lambda^* \times \Lambda^*$ we may define an element $\pi_A \in M$ where $$\pi_A = ( \mathds{1}_A, \lambda_1, \lambda_2)$$ where $\lambda_1, \lambda_2$ are arbitrary. Define also $$L_{i,j} = \{ (w_1,w_2) \in \Lambda^* \times \Lambda^* ~|~ |w_1| = i, |w_2| = j \}$$ where $|w|$ denotes the length of $w \in \Lambda^*$. For $S \subset \Lambda^* \times \Lambda^*$ we may define a density $$d_N(S) =  \frac{1}{N^2} \sum_{i,j=0}^{N-1} \frac{|S \cap L_{i,j}|}{|L_{i,j}|}.$$

Now let $$\mu_N = \frac{1}{N^2}\sum_{i,j=0}^{N-1} (P_{1}^*)^{i} (P_{2}^*)^{j} \delta_{\pi_A}$$ and let $$E=\{ (\tau, \lambda_1, \lambda_2) \in M ~|~ \tau(\emptyset, \emptyset) = 1 \}.$$

\begin{prop} For positive integers $N$ we have that $$\mu_N(E) =d_N(A).$$

\end{prop}

\begin{proof} First note that for any measure $\mu$ on $M$ we have that 
\begin{align*} (P_{1}^*)^{n} (P_{2}^*)^{m} \mu &= (\sum_{\lambda \in \Lambda}\frac{1}{|\Lambda|}(T_1)^*_{\lambda})^n(\sum_{\lambda \in \Lambda}\frac{1}{|\Lambda|}(T_2)^*_{\lambda})^m \mu \\ &= \frac{1}{|\Lambda|^{n+m}} \sum_{w_1 \in \Lambda^n}\sum_{w_2 \in \Lambda^m} \mu X_{w_1}^* Y_{w_2}^* \end{align*} 
where we have defined $\mu X_w^*$ to be the pushforward measure of $\mu$ under the mapping $\pi \mapsto  \pi X_w$, and likewise we have similarly defined $\mu Y_w^*$. However, note that $$(\delta_{\pi_A}X_{w_1}^* Y_{w_2}^*) (E) = \delta_{\pi_A X_{w_1}Y_{w_2}}(E) $$ and so this equals $1$ if and only if $(\mathds{1}_A X_{w_1}Y_{w_2})(\emptyset, \emptyset) = 1$ which happens if and only if $\mathds{1}_A(w_1,w_2) \in 1$ which happens if and only if $(w_1,w_2) \in A$. Thus  $$(P_{1}^*)^{n} (P_{2}^*)^{m}\delta_{\pi_A}(E) =  \frac{|A \cap L_{n,m}|}{|L_{n,m}|}$$ as $|L_{n,m}| = |\Lambda|^{n+m}$, which completes the proof.\end{proof}

We now let $\mu$ be a weak$^*$ limit of the $\mu_N$ such that $$\mu(E) = \limsup_{N \to \infty} \mu_N(E) = \limsup_{N \to \infty} d_N(A) = \overline{d}(A)>0,$$ where we have used the fact that $E$ is clopen. Thus $\mu$ is a stationary measure for the commuting systems $(M,T_1, p)$ and $(M,T_2, p)$. To ensure that we can apply Theorem~\ref{thm: arithmetic product trees in Markov} we now need to show the following lemma.

\begin{lemma} Both of these systems have disjoint images. \end{lemma}

\begin{proof} For $\lambda$ and $x\in M$ we have $(T_1)_{\lambda}(x) \in \mathcal{T} \times \{\lambda\} \times \Lambda$. But $ \mathcal{T} \times \{\lambda\} \times \Lambda$ are disjoint subsets of $M$ for different $\lambda$, thus $(M,T_1,p)$ has disjoint images (by definition). Likewise $(T_2)_{\lambda} (x) \in \mathcal{T} \times \Lambda \times \{\lambda\}$ hence $T_2$ also has disjoint images. \end{proof}

It will be convenient to write $$\pi W_{\alpha} = \pi X_{\alpha_1} Y_{\alpha_2} \quad \text{for } \pi \in M \text{ and } \alpha = (\alpha_1,\alpha_2) \in \Lambda^* \times \Lambda^*.$$ 
Now fix an integer $r>0$ and $u,v \in \bZ_{>0}^2$. Fixing $u,v \in \bZ_{>0}^2$ and $r \in \bZ_{>0}$, by Theorem~\ref{thm: arithmetic product trees in Markov} we have an integer $n$ and $E' \subset E$ with $\mu(E')>0$ such that each $\pi \in E'$ is the root of some $(nu,nv)$-arithmetic product tree of order $r$ in $E$ (with respect to our pair of Markov systems $(M,T_1,p)$, $(M,T_2,p)$). By construction, $\mu$ is supported on the closure of orbit of $\pi_A$, where we define the orbit to be the set $$\{ \pi_A W_{\alpha} ~|~ \alpha \in \Lambda^* \times \Lambda^* \}.$$ Thus we now fix such a root $\pi \in E'$ that is also in the closure of this orbit. Now observe that if $x_0,\ldots, x_{\ell} \in M$ a path of length $\ell$ with initial direction $\lambda_0$ in the system $(M,T_1,p)$ then that means that $x_i = x_0 X_{\lambda_0 \ldots \lambda_{i-1}}$. Consequently, this means that a $(nu,nv)$-arithmetic tree in $E$ with root $\pi$ is a map $\psi:B_r(\widetilde{\Gamma}) \to E$ that can be written as $$\psi(\gamma) = \pi W_{\phi(\gamma)}$$ where $\phi:B_r(\widetilde{\Gamma}) \to \Lambda^* \times \Lambda^*$ is a tree in $\Lambda^* \times \Lambda^*$ as defined in Section~\ref{section: Free Products}. Now since $\pi$ is in the orbit closure of $\pi_A$ and the image of $\phi$ is finite (as its domain is), this means that we can choose appropriate $\alpha \in \Lambda^* \times \Lambda^*$ such that $$\pi(w) = \pi_A W_{\alpha} (w) \quad \text{for all } w \in \phi(B_r(\widetilde{\Gamma})).$$ Now for all $\gamma \in B_r(\widetilde{\Gamma})$ we have that $\psi(\gamma) = \pi W_{\phi(\gamma)} \in E$ and thus $$1 = \pi W_{\phi(\gamma)} (\emptyset) = \pi(\phi(\gamma)) = \pi_A W_{\alpha} (\phi(\gamma)) = \pi_A(\alpha \phi(\gamma)),$$ which means that $\alpha \phi(\gamma) \in A$. But one easily sees that $\gamma \mapsto \alpha \phi(\gamma)$ is also an $(nu,nv)$-arithmetic tree as $\phi$ is (the tree structure is preserved by left multiplication). This completes the proof that if $\overline{d}(A)>0$ then $A$ contains an $(nu,nv)$-arithmetic product tree of any order, which is exactly Theorem~\ref{thm: main thm on trees in A}.

\end{document}